\documentclass{elsarticle}
\usepackage{amsfonts}
\usepackage{amsmath}
\usepackage[mathscr]{eucal}
\usepackage{amssymb}
\usepackage{latexsym}
\usepackage{amsthm}
\usepackage{amscd}
\usepackage{epsf}
\usepackage{graphicx}
\usepackage{makeidx}
\usepackage{enumerate}

\newtheorem{theorem}{Theorem}[section]
\newtheorem{proposition}[theorem]{Proposition}
\newtheorem{corollary}[theorem]{Corollary}
\newtheorem{lemma}[theorem]{Lemma}

\newtheorem{definition}[theorem]{Definition}

\newtheorem{remark}[theorem]{Remark}

\def\bN{\mathbb N_0}

\def\bR{{\mathbb R}}
\def\bC{{\mathbb C}}

\def\F2{{\mathbb F}_2}

\def\bsx{{\boldsymbol{x}}}

\def\bsk{{\boldsymbol{k}}}

\newcommand{\nat}{\mathbb{N}_0}
\newcommand{\wal}{\mathrm{wal}}

\newcommand{\omebar}{\overline{\omega}_b}
\newcommand{\W}[2]{W(#1)(#2)}
\newcommand{\I}[1]{I(#1)}
\newcommand{\Wnorm}[1]{W(#1)(\cdot)}
\newcommand{\Wnormof}[2]{\| \Wnorm{#1} \|_{L^{#2}}}
\newcommand{\Wextra}[3]{W_{#1}(#2)(#3)}
\newcommand{\Iextra}[2]{I_{#1}(#2)}
\newcommand{\Wextranorm}[2]{W_{#1}(#2)(\cdot)}
\newcommand{\Wextranormof}[3]{\|\Wextranorm{#1}{#2}\|_{L^{#3}}}
\newcommand{\summore}[1]{k^{#1}_{>}}
\newcommand{\sumless}[1]{k^{#1}_{\leq}}
\newcommand{\summorej}[1]{{k_j}^{#1}_{>}}
\newcommand{\sumlessj}[1]{{k_j}^{#1}_{\leq}}
\newcommand{\mb}{m_b}
\newcommand{\Mb}{M_b}
\newcommand{\Cv}[1][v]{\frac{b\mb}{b-\Mb}\left( 1-\left(\frac{\Mb}{b}\right)^{#1} \right)}
\newcommand{\Sob}{\mathcal{H}_\alpha}
\newcommand{\Sobnorm}[1]{\|#1\|_{\mathrm{Sob, \alpha}}}
\newcommand{\innerprod}[2]{\langle #1, #2 \rangle_\alpha}
\newcommand{\muper}[2]{\mu_{#1, \mathrm{per}}(#2)}
\newcommand{\Sobper}{\mathcal{H}_{\alpha, \mathrm{per}}}
\newcommand{\innerprodper}[2]{\langle #1, #2 \rangle_{\alpha, \mathrm{per}}}
\newcommand{\antiW}[3]{\overline{\wal_{#2}^{[#1]}(#3)}}

\newenvironment{enuroman}{\begin{enumerate}[\normalfont (i)]}{\end{enumerate}}

\begin{document}

\begin{frontmatter}
\title{Formulas for the Walsh coefficients of smooth functions
and their application to bounds on
the Walsh coefficients 
}
\author[Tokyo]{Kosuke Suzuki\corref{cor1}\fnref{fn1}}
\ead{ksuzuki@ms.u-tokyo.ac.jp}
\fntext[fn1]{JSPS research fellow, phone: +81-3-5465-8283}

\author[Tokyo]{Takehito Yoshiki}
\ead{yosiki@ms.u-tokyo.ac.jp}

\address[Tokyo]{Graduate School of Mathematical Sciences, The University of Tokyo, 3-8-1 Komaba, Meguro-ku, Tokyo 153-8914 Japan}
\cortext[cor1]{Corresponding author}



\begin{abstract}
We establish formulas for the $b$-adic Walsh coefficients of
functions in $C^\alpha[0,1]$ for an integer $\alpha \geq 1$
and give upper bounds on the Walsh coefficients of these functions.
We also study the Walsh coefficients of
periodic and non-periodic functions in reproducing kernel Hilbert spaces.
\end{abstract}


\end{frontmatter}

\section{Introduction}\label{sec:Introduction}
The Walsh coefficients of a function are the generalized Fourier coefficients for the Walsh system,
which is a normal orthogonal system.
It is often used instead of the trigonometric Fourier system for analyzing numerical integration
\cite{Dick2010dna}, approximation \cite{Dick2008afu, Baldeaux2009asf}
and constructing low-discrepancy point sets \cite{Dick2014dbi, Dick2014odb},
especially when we consider point sets so-called digital nets,
which have the suitable group structure for the Walsh system. 
In particular, the decay of the Walsh coefficients of smooth functions is fundamental
to analyze such problems for spaces of smooth functions.
For example, It is used to give explicit constructions of quasi-Monte Carlo rules
which achieve the optimal rate of convergence for smooth functions in \cite{Dick2007ecq, Dick2008wsc}
and algorithm to approximate functions in Sobolev spaces \cite{Baldeaux2009asf}.
In this paper, we develop the theory of the decay of the Walsh coefficients of smooth functions.

Throughout the paper we use the following notation:
Assume that $b \geq 2$ is a positive integer.
We assume that $k$ is a nonnegative integer whose $b$-adic expansion is
$k = \kappa_1b^{a_1-1} + \cdots + \kappa_v b^{a_v - 1}$
where $\kappa_i$ and $a_i$ are integers with $0 < \kappa_i \leq b-1$, $a_1 > \cdots > a_v \geq 1$.
For $k=0$ we assume that $v=0$ and $a_0=0$.
We denote by $\bN$ the set of nonnegative integers.
Let $\omega_b := \exp(2\pi \sqrt{-1}/b)$.

The Walsh functions were first introduced by Walsh \cite{Walsh1923csn},
see also \cite{Fine1949wf, Chrestenson1955cgw}.
For $k \in \bN$, the $b$-adic $k$-th Walsh function $\wal_{k}(\cdot)$
is defined as
$$
\wal_k(x) := \omega_b^{\sum_{i=1}^{v} \kappa_i \xi_{a_i}},
$$
for $x\in [0,1)$ whose $b$-adic expansion is given by
$x = \xi_1b^{-1} + \xi_2b^{-2} + \cdots$,
which is unique in the sense that infinitely many of the digits $\xi_i$ are different from $b-1$.
We also consider $s$-dimensional Walsh functions.
For $\bsk = (k_1, \dots, k_s) \in \bN^s$ and $\bsx = (x_1, \dots, x_s) \in [0,1)^s$, 
the $b$-adic $\bsk$-th Walsh function $\wal_{\bsk}(\cdot)$
is defined as
$$
\wal_{\bsk}(\bsx) := \prod_{j=1}^s \wal_{k_j}(x_j).
$$
For $\bsk \in \bN^s$ and $f \colon [0,1)^s \to \bC$,
we define the $\bsk$-th Walsh coefficient of $f$ as
  \begin{align*}
     \widehat{f}(\bsk) := \int_{[0,1)^s} f(\bsx)\overline{\wal_{\bsk}(\bsx)} \, d\bsx.
  \end{align*}
It is well-known that
the Walsh system $\{\wal_{\bsk}(\cdot) \mid \bsk\in \nat^s\}$ is a complete orthonormal system in $L^2[0,1)^s$ for any positive integer $s$
(for a proof, see e.g., \cite[Theorem A.11]{Dick2010dna}).
Hence we have a Walsh series expansion
\begin{align*}
f(\bsx) \sim \sum_{\bsk\in \nat^s}\widehat{f}(\bsk)\wal_{\bsk}(\bsx)
\end{align*}
for any $f\in L^2[0,1)^s$.
Let $s=1$ at this moment.
It is known that if $f \in C^0[0,1]$ has bounded variation then $f$ is equal to its Walsh series expansion,
that is, for all $x \in [0,1)$ we have
$f(x) = \sum_{k \in \nat}\widehat{f}(k)\wal_{k}(x)$,
see \cite{Walsh1923csn}.
More information on Walsh analysis can be found in the books
\cite{Schipp1990ws, Golubov1991wsa}.



There are several studies on the decay of the Walsh coefficients.
Fine considered
the Walsh coefficients of functions which satisfy a H\"{o}lder condition in \cite{Fine1949wf}.
Dick studied \cite{Dick2007ecq, Dick2008wsc} the decay of the Walsh coefficients of functions of smoothness 
$\alpha \geq 1$ and in more detail in \cite{Dick2009dwc}:
It was proved that if a function $f$ has $\alpha-1$ derivatives
for which $f^{(\alpha-1)}$ satisfies a Lipschitz condition,
then $|\widehat{f}(k)| \leq C b^{-\mu_\alpha(k)}$ for all $k$,
where $C$ is a positive real number which is independent of $k$
and $\mu_\alpha(k) := a_1 + \dots + a_{\min(\alpha,v)}$.
Dick also proved that this order is the best possible.
That is,
for $f$ of smoothness $\alpha$, if there exists $1 \leq r \leq \alpha$
such that $\widehat{f}(k)$ decays faster than $b^{-a_1 - \dots - a_r}$ for all $k \in \bN$ and $v \geq r$,
then $f$ is a polynomial of degree at most $r-1$
\cite[Theorem~20]{Dick2009dwc}.

Recently, Yoshiki gave a method to analyze the dyadic (i.e., 2-adic) Walsh coefficients
in \cite{Yoshiki}.
He introduced dyadic differences of (maybe discontinuous) functions
and gave a formula in which the dyadic Walsh coefficients are
given by dyadic differences multiplied by constants.
Dyadic differences of a smooth function are expressed
in terms of derivatives of the function.
This enabled him to establish a formula for the dyadic Walsh coefficients
of smooth functions expressed in terms of those derivatives as
\begin{equation}\label{eq:formula-intro}
\widehat{f}(k) = (-1)^v \int_0^1 f^{(v)}(x) \W{k}{x} \, dx,
\end{equation}
where $\Wnorm{k} \colon [0,1] \to \bC$
is given by the iterated integral of functions derived from dyadic differences
(note that the notation $\Wnorm{k}$ coincides with that in \cite{Yoshiki} up to constant multiple).
From this formula,
he obtained a bound on the dyadic Walsh coefficients for $\alpha$ times continuously differentiable functions for $\alpha \geq 1$.

In this paper, we first generalize \eqref{eq:formula-intro} for the $b$-adic Walsh coefficients
in Theorems~\ref{thm:Walsh_formula} and \ref{thm:Walsh_formula-multi}.
Although our generalization focuses only on smooth functions,
it can treat the $b$-adic case which is not included in \cite{Yoshiki}.
Our proof technique is completely different from that in \cite{Yoshiki}.
We prove that $\Wnorm{k}$ is also given by the following two forms:
the $v$-th anti-derivative of $\overline{\wal_{k}(\cdot)}$
and
the iterated integral of Walsh functions starting from the highest frequency
and adding in lower frequencies at each subsequent step as in Definition~\ref{def:W}.
Then, using the latter form of $\Wnorm{k}$, we give bounds on the $b$-adic Walsh coefficients
for $\alpha$ times continuously differentiable functions as
\begin{align}\label{eq:bound-smooth-intro}
|\widehat{f}(k)|
\leq c(\alpha,v,b) \|f^{(\min(\alpha, v))}\|_{L^1}\frac{b^{-\mu_\alpha(k)}}{\mb^{\min(\alpha,v)}},
\end{align}
where $\mb := 2\sin(\pi/b)$ and
$c(\alpha,v,b)$ is an explicit positive constant, see Theorem~\ref{thm:Walsh-bound-normal}.
We note that $c(\alpha,v,b)$ is bounded with respect to $\alpha$ and $v$.
In the dyadic case, we also have similar bounds in terms of the $L^p$-norm of the derivatives of $f$
instead of the $L^1$-norm and this matches the bound in \cite{Yoshiki}.
This bound is extended to the multivariate case in Theorem~\ref{thm:Walsh-bound-normal-s}
and the case $\alpha$ is infinity in Corollary~\ref{cor:Walsh-bound-infty}.


Furthermore, we give improved bounds on the $b$-adic Walsh coefficients
for periodic and non-periodic functions in the Sobolev spaces $\Sobper$ and $\Sob$
which are considered in \cite{Dick2009dwc}.
In \cite{Dick2009dwc},
Dick gave bounds on the Walsh coefficients of a polynomial $b_r(\cdot)$,
which is equal to the Bernoulli polynomial $B_r(\cdot)$ up to constant multiple, by
\begin{align}\label{eq:bound-ber-intro}
|\widehat{b_r}(k)|
\leq C_{\mathrm{Ber}}(r,v,b) \frac{b^{-\muper{r}{k}}}{\mb^r},
\end{align}
where $C_{\mathrm{Ber}}(r,v,b)$ is a positive constant and
$\muper{r}{k}$ is defined as in \eqref{eq:def-muper}.
Using this, he obtained a bound on the Walsh coefficients for $f \in \Sob$ by
\begin{align}\label{eq:bound-Sob-intro}
|\widehat{f}(k)| \leq b^{-\mu_{\alpha}(k)} C_{b, \alpha, q} \|f\|_{p, \alpha},
\end{align}
where $1 \leq p, q \leq \infty$ are real numbers with $1/p + 1/q =1$,
$\|f\|_{p, \alpha}$ is a norm related to $\Sob$ and $C_{b, \alpha, q}$ is a positive constant,
see \cite[Corollary~14.22]{Dick2010dna}.

In this paper, we improve the constants $C_{\mathrm{Ber}}(r,v,b)$ in \eqref{eq:bound-ber-intro}
and $C_{b, \alpha, q}$ in \eqref{eq:bound-Sob-intro} in many cases,
see Theorem~\ref{thm;b_r(k)} and Corollary~\ref{cor:bound-Sob-with-norm}, respectively.
This is done by using another formula for the Walsh coefficients
where higher derivatives $f^{(i)}$ for $i \geq v$ appear which is established in Theorem~\ref{thm:extra_exp}.
Although $C_{\mathrm{Ber}}(r,v,b)$ in \cite{Dick2009dwc} depends exponentially on $r$,
our constant is independent of $r$ and
bounded with respect to $v$,
Further, if $b=2$, our constant is strictly better for all $r$ and $v$.
Hence, in many cases, our constant is better.
Using this improved bound, we give a new constant for \eqref{eq:bound-Sob-intro}.
Again our constant is improved in many cases including the case $b=2$,
where it is strictly better for all $\alpha$.
We also have similar improvement for $\Sobper$, see Theorem~\ref{thm:bound-Sobper}.


The rest of the paper is organized as follows.
We give two formulas for the Walsh coefficients of smooth functions
in Sections~\ref{sec:formula} and \ref{sec:Another-formula}.
Bounds on the Walsh coefficients of smooth functions and Bernoulli polynomials
are given in Sections~\ref{sec:coeff-of-smooth} and \ref{sec:coeff-of-Bernoulli},
respectively.
In Section~\ref{sec:Sobolev} (resp.\ Section~\ref{sec:periodic}),
we give a bound on the Walsh coefficients of functions in 
non-periodic (resp.\ periodic) reproducing kernel Hilbert spaces.

\section{Integral formula for the Walsh coefficients of smooth functions}\label{sec:formula}
We introduce further notation which is used throughout the paper.
For $k>0$, let $k' = k - \kappa_vb^{a_v-1}$
(this notation is different from that in \cite{Dick2009dwc}.
we remove the smallest term from the expansion of $k$).
Let $v(k) := v$ be the Hamming weight of $k$ with respect to its $b$-adic expansion,
i.e., the number of non-zero digits of $k$.

In this section, we define the function $\Wnorm{k}$
and establish a formula in which the Walsh coefficients of smooth functions
are expressed in terms of $\Wnorm{k}$ and derivatives of the functions.

\begin{definition}\label{def:W}
For $k \in \bN$, we define functions
$\Wnorm{k} \colon [0,1] \to \bC$ recursively as
\begin{align*}
\W{0}{x} &:= 1,\\
\W{k}{x} &:= \int_0^x \overline{\wal_{\kappa_v b^{a_v- 1}}(y)} \W{k'}{y}\, dy,
\end{align*}
and the integral value of $\Wnorm{k}$ as
$$
\I{k} := \int_0^1 \W{k}{x} \, dx.
$$
\end{definition}

By definition, $\Wnorm{k}$ is continuous for all $k \in \bN$.
Note that we have
$$
\W{k}{x} = \int_0^x \W{k'}{y}\, dy \qquad \text{for $x \in [0, b^{-a_v}]$}
$$
since we have $\wal_{\kappa_v b^{a_v- 1}}(y) = 1$ for all $y \in [0, b^{-a_v})$.
We show the periodicity of $\Wnorm{k}$ in the next lemma.

\begin{lemma}\label{lem:Wperiod}
Let $k \in \bN$.
Let $x \in [0,1)$ and
$x = cb^{-a_v} + x'$, where $0 \leq c < b^{a_v}$ is an integer
and $0 \leq x' < b^{-a_v}$ is a real number.
Then we have
$$
\W{k}{x} = \frac{1-\omebar^{c\kappa_v}}{1-\omebar^{\kappa_v}} \W{k}{b^{-a_v}}
+ \omebar^{c\kappa_v}\W{k}{x'}.
$$
In particular, $\Wnorm{k}$ is a periodic function with period $b^{-a_v + 1}$ if $v>0$.
\end{lemma}

\begin{proof}
We prove the lemma by induction on $v$.
If $v=0$, trivially the result holds since for $x \in [0,1)$ we always have that $c=0$ and $x' = x$.
Hence we now assume that the claim holds for $v-1$.
Since $v(k') = v-1$, we can apply the induction assumption for $v-1$ to $\Wnorm{k'}$.
Hence $\Wnorm{k'}$ is periodic with period $b^{-a_{v-1}+1}$
and in particular with period $b^{-a_v}$ if $v>1$,
and $\Wnorm{k'}$ is constant if $v=1$.
Hence we have
\begin{align*}
\W{k}{x} &= \sum_{i=0}^{c-1} \int_{ib^{-a_v}}^{(i+1)b^{-a_v}} \overline{\wal_{\kappa_v b^{a_v -1}}(y)} \W{k'}{y}\, dy\\
&\qquad + \int_{cb^{-a_v}}^{cb^{-a_v} + x'} \overline{\wal_{\kappa_v b^{a_v -1}}(y)} \W{k'}{y}\, dy\\
&= \sum_{i=0}^{c-1} \omebar^{i \kappa_v} \int_0^{b^{-a_v}} \W{k'}{y}\, dy
 + \omebar^{c \kappa_v} \int_0^{x'} \W{k'}{y}\, dy\\
&= \frac{1-\omebar^{c\kappa_v}}{1-\omebar^{\kappa_v}} \W{k}{b^{-a_v}}
+ \omebar^{c\kappa_v}\W{k}{x'},
\end{align*}
which proves the first claim for $v$.
Further, the most right-hand side of the above does not change if one changes $c$ to $c+b$.
Hence $\Wnorm{k}$ is periodic with period $b^{-a_v+1}$, which proves the second claim for $v$.
\end{proof}

We also need a lemma which gives anti-derivatives of $\wal_{k}(\cdot)$.
For $n, k \in \bN$, we define two symbols
$\summore{n}$ and $\sumless{n}$ as
$\summore{n} := \sum_{i = n+1}^v \kappa_i b^{a_i - 1}$ and
$\sumless{n} := \sum_{i = 1}^{\min(n, v)} \kappa_i b^{a_i - 1}$, respectively.
Note that $\sumless{n} + \summore{n} = k$.
\begin{lemma}\label{lem:anti-derivatives-of-W}
Let $k \in \bN$. For $n \in \bN$, define functions $\antiW{n}{k}{\cdot} \colon [0,1] \to \bC$ 
recursively as
$\antiW{0}{k}{x} := \overline{\wal_{k}(x)}$ and $\antiW{n}{k}{x} := \int_0^x \antiW{n-1}{k}{y} \, dy$.
Then for all $0 \leq n \leq v$ we have
\[
\antiW{n}{k}{x} = \overline{\wal_{\summore{n}}(x)} \W{\sumless{n}}{x}  \quad \text{for all $x \in [0,1)$}.
\]
Further, $\antiW{n}{k}{1} = 0$ for all $1 \leq n \leq v$.
\end{lemma}

\begin{proof}
It suffices to show that for all integer $1 \leq n \leq v$
\[
\int_0^x \overline{\wal_{\summore{n-1}}(y)} \W{\sumless{n-1}}{y} \,dy
= \overline{\wal_{\summore{n}}(x)} \W{\sumless{n}}{x} \quad \text{for all $x \in [0,1)$}
\]
and that $\antiW{n}{k}{1} = 0$ for all $1 \leq n \leq v$.

We show the first claim.
Let $1 \leq n \leq v$ an integer, $x \in [0,1)$ and
$x = cb^{-a_n+1} + x'$, where $0 \leq c < b^{a_n-1}$ is an integer
and $0 \leq x' < b^{-a_n+1}$ is a real number.
Note that we have
$
\wal_{\summore{n-1}}(y)
= \wal_{\summore{n}}(y) \wal_{\kappa_n b^{a_n - 1}}(y)
$
for all $y \in [0,1)$ and
$$
\wal_{\summore{n}}(y) = \wal_{\summore{n}}(ib^{-a_n+1})
\text{\qquad for $y \in [ib^{-a_n+1}, (i+1)b^{-a_n+1})$}
$$ 
for each integer $0 \leq i < b^{a_n-1}$.
Hence we have
\begin{align*}
&\int_0^x \overline{\wal_{\summore{n-1}}(y)} \W{\sumless{n-1}}{y} \,dy\\
&= \sum_{i=0}^{c-1}\overline{\wal_{\summore{n}}(ib^{-a_n+1})}\int_{ib^{-a_n+1}}^{(i+1)b^{-a_n+1}} \overline{\wal_{\kappa_n b^{a_n - 1}}(y)}\W{\sumless{n-1}}{y} \, dy\\
&\quad + \overline{\wal_{\summore{n}}(cb^{-a_n+1})}\int_{cb^{-a_n+1}}^{x} \overline{\wal_{\kappa_n b^{a_n - 1}}(y)}\W{\sumless{n-1}}{y} \, dy\\
&= \sum_{i=0}^{c-1}\overline{\wal_{\summore{n}}(ib^{-a_n+1})} \left[\W{\sumless{n}}{y}\right]_{ib^{-a_n+1}}^{(i+1)b^{-a_n+1}} + \overline{\wal_{\summore{n}}(x)} \left[\W{\sumless{n}}{y}\right]_{cb^{-a_n+1}}^{x},
\end{align*}
where we use $\wal_{\summore{n}}(cb^{-a_n+1}) = \wal_{\summore{n}}(x)$ in the last equality.
The first term of the most right-hand side of the above is equal to zero
since $\W{\sumless{n}}{ib^{-a_n+1}} = \W{\sumless{n}}{(i+1)b^{-a_n+1}} = 0$
by Lemma~\ref{lem:Wperiod}.
Similarly, the second term is equal to $\overline{\wal_{\summore{n}}(x)}\W{\sumless{n}}{x}$.
This proves the result for $x \in [0,1)$.

We now show the second claim. Considering the above calculation for $c = b^{a_n-1}$ and $x' = 0$,
we obtain $\antiW{n}{k}{1} = 0$ for all $1 \leq n \leq v$.
\end{proof}

Now, by using integral-by-parts and  Lemma~\ref{lem:anti-derivatives-of-W} iteratively, it is easy to show the following formula for the Walsh coefficients.

\begin{theorem} \label{thm:Walsh_formula}
Let $k \in \bN$.
Assume that $f \in C^\alpha[0,1]$ for a positive integer $\alpha$.
Then for an integer $0 \leq n \le \min(\alpha, v)$ we have
\begin{align*}
\widehat{f}(k)
=  (-1)^{n} \int_0^1 f^{(n)}(x) \antiW{n}{k}{x} \, dx
= (-1)^{n} \int_0^1 f^{(n)}(x) \overline{\wal_{\summore{n}}(x)} \W{\sumless{n}}{x} \, dx.
\end{align*}
\end{theorem}
We remark that Theorem~\ref{thm:Walsh_formula} for $n = v$ gives Formula \eqref{eq:formula-intro}
announced in Introduction.


Now we consider the $s$-variate case.
For a function $f \colon [0,1)^s \to \bR$,
let $f^{(n_1, \dots, n_s)} := (\partial/\partial x_1)^{n_1}\cdots(\partial/\partial x_s)^{n_s}f$  
be the $(n_1, \dots, n_s)$-th derivative of $f$.
Considering coordinate-wise integration, we have the following.
\begin{theorem}\label{thm:Walsh_formula-multi}
Let $\bsk = (k_1, \dots, k_s) \in \bN^s$.
Assume that $f \colon [0,1]^s \to \bR$
has continuous mixed partial derivatives up to order $\alpha_j$ in each variable $x_j$.
Let $n_j$ be integers with
$0 \leq n_j \leq \min(\alpha_j, v(k_j))$ for $1 \leq j \leq s$.
Then we have
$$
\widehat{f}(\bsk)
= (-1)^{n_1 + \dots + n_s} \int_{[0,1)^s} f^{(n_1, \dots, n_s)}(\bsx)
\prod_{j=1}^s \overline{\wal_{\summorej{n_j}}(x_j)} \W{\sumlessj{n_j}}{x_j} \, d\bsx.
$$
\end{theorem}

\section{The Walsh coefficients of smooth functions}\label{sec:coeff-of-smooth}
Let $f \in C^\alpha[0,1]$ and
$p, q \in [1, \infty]$ with $1/p + 1/q = 1$.
By Theorem~\ref{thm:Walsh_formula}
for $n = \min(\alpha, v)$ and H\"{o}lder's inequality, we have
\begin{align}\label{eq:Walsh-bound-smooth}
|\widehat{f}(k)|
&\leq \int_0^1 \left|f^{(\min(\alpha, v))}(x) \overline{\wal_{\summore{\min(\alpha, v)}}(x)} \W{\sumless{\min(\alpha, v)}}{x} \right| \, dx \notag \\
&\leq \|f^{(\min(\alpha, v))}\|_{L^p} \Wnormof{\sumless{\alpha}}{q}.
\end{align}
Thus, it suffices to bound $\Wnormof{\sumless{\alpha}}{q}$
to bound $|\widehat{f}(k)|$.
We give bounds on
$\Wnormof{\sumless{\alpha}}{\infty}$ for the non-dyadic case,
$\Wnormof{\sumless{\alpha}}{q}$ for the dyadic case
and $|\widehat{f}(k)|$ in Sections~\ref{subsec:non-dyadic-W-bound},
\ref{subsec:dyadic-W-bound} and \ref{subsec:bound-smooth}, respectively.

We introduce a function $\mu$ as follows.
For $k \in \bN$, we define
\begin{equation}
\mu(k)
:= \begin{cases}
0 & \text{for $k=0$}, \\
a_1 + \dots + a_v & \text{for $k \neq 0$}.
\end{cases}
\end{equation}
For $\bsk = (k_1, \cdots, k_s) \in \bN^s$, we define
$\mu(\bsk) := \sum_{j=1}^s \mu(k_j)$.

For subsequent analysis,
we give the exact values of $\I{k}$ and $\W{k}{b^{-a_v}}$ in the next lemma.
\begin{lemma}\label{lem:I-value}
For $k \in \bN$, we have the following.
\begin{enuroman}
\item \label{I-value1}
$\displaystyle \I{k} = \frac{b^{-\mu(k)}}{\prod_{i=1}^v (1-\omebar^{\kappa_i})}$,
\item \label{I-value2}
$\displaystyle
\W{k}{b^{-a_v}} = \frac{b^{-\mu(k)}}{\prod_{i=1}^{v-1} (1-\omebar^{\kappa_i})}$.
\item \label{I-value3}
Let $x \in [0,1)$ and $x = cb^{-a_v}+ x'$
where $0 \leq c < b^{a_v}$ is an integer and $0 \leq x' < b^{-a_v}$ is a real number.
Then we have
$$
\W{k}{x} = (1-\omebar^{c\kappa_v}) \I{k}
+ \omebar^{c\kappa_v}\W{k}{x'}.$$
\end{enuroman}
Here, the empty products $\prod_{i=1}^0$ and $\prod_{i=1}^{-1}$
are defined to be $1$.
\end{lemma}

\begin{proof}
By Lemma~\ref{lem:Wperiod} we have
\begin{align} \label{eq:I-to-W}
\I{k}
&= \sum_{i=0}^{b^{a_v}-1} \int_{ib^{-a_v}}^{(i+1)b^{-a_v}} \W{k}{x} \, dx \notag \\
&= \sum_{i=0}^{b^{a_v}-1} \int_{0}^{b^{-a_v}}
 \left(\frac{1-\omebar^{i\kappa_v}}{1-\omebar^{\kappa_v}}\W{k}{b^{-a_v}} + \omebar^{i  \kappa_v}\W{k}{x}\right) dx \notag \\
&= \frac{\W{k}{b^{-a_v}}}{1-\omebar^{\kappa_v}} b^{-a_v} \sum_{i=0}^{b^{a_v}-1}(1 - \omebar^{i \kappa_v})
 + \sum_{i=0}^{b^{a_v}- 1} \omebar^{i \kappa_v} \int_0^{b^{-a_v}}\W{k}{x} \, dx \notag \\
&= \frac{\W{k}{b^{-a_v}}}{1-\omebar^{\kappa_v}}.
\end{align}

Furthermore, $\W{k}{b^{-a_v}}$ is computed as
\begin{align} \label{eq:W-to-I}
\W{k}{b^{-a_v}}
&= \int_0^{b^{-a_v}} \W{k'}{x}  \, dx \notag \\
&= b^{-a_v} \I{k'}, 
\end{align}
where we use the fact that $\Wnorm{k'}$ is periodic with period $b^{-a_v}$,
which follows from Lemma~\ref{lem:Wperiod}, in the last equality.
Using equations \eqref{eq:I-to-W} and \eqref{eq:W-to-I} iteratively,
we have \eqref{I-value1} and \eqref{I-value2}.
Combining \eqref{eq:I-to-W} and Lemma~\ref{lem:Wperiod},
we have \eqref{I-value3}.
\end{proof}

In the following, we consider two cases
in order to bound
$\Wnormof{k}{\infty}$:
the non-dyadic case and the dyadic case.
We define two positive constants $\mb$ and $\Mb$ as
\begin{align*}
\mb
&:= \min_{c = 1, 2, \dots, b-1} |1-\omebar^c| = 2\sin(\pi/b),\\
\Mb
&:= \max_{c = 1, 2, \dots, b-1} |1-\omebar^c| =
\begin{cases}
2 &  \text{\quad if $b$ is  even},\\
2\sin((b+1)\pi/2b) &  \text{\quad if $b$ is odd}.
\end{cases}
\end{align*}

\subsection{Non-dyadic case}\label{subsec:non-dyadic-W-bound}
The following lemmas are needed to bound $\sup_{x' \in [0,b^{-a_v}]} |\W{k}{x'}|$.
\begin{lemma}\label{lem:convex-endpoint}
Let $A, B$ be complex numbers and $r$ be a positive real number.
Then we have $\sup_{x \in [0,r]}|Ax + B| = \max(|B|, |rA + B|)$.
\end{lemma}

\begin{proof}
We have
$$
\sup_{x \in [0,r]}|Ax + B|
= \sqrt{\sup_{x \in [0,r]}|Ax + B|^2}
= \sqrt{\sup_{x \in [0,r]}(|A|^2 x^2 + 2 \mathrm{Re}(A\overline{B})x + |B|^2)}.
$$
Since $|A|^2 x^2 + 2\mathrm{Re}(A\overline{B})x + |B|^2$
is a convex function on $[0, r]$,
its maximum value occurs at its endpoints. 
\end{proof}

\begin{lemma}\label{lem:first-term2}
Let $a$ and $1 \leq \kappa \leq b-1$ be positive integers.
Then we have
$$
\sup_{c' = 0, 1, \dots, ab} \left|\sum_{i=0}^{c'-1} (1-\omebar^{i\kappa})\right|
\leq ab.
$$
\end{lemma}

\begin{proof}
Since $\sum_{i=0}^{ab-1} \omebar^{i\kappa} = 0$, we have
\begin{align*}
\sup_{c' = 0, 1, \dots, ab} \left|\sum_{i=0}^{c'-1} (1-\omebar^{i\kappa})\right|
&= \sup_{c' = 0, 1, \dots, ab}\left|c' + \sum_{i=c'}^{ab-1} \omebar^{i\kappa}\right|\\
&\leq \sup_{c' = 0, 1, \dots, ab}\left(c' +\sum_{i=c'}^{ab-1} \left|\omebar^{i\kappa}\right|\right)
= ab. \qedhere
\end{align*}
\end{proof}

We now have an upper bound on $\sup_{x' \in [0,b^{-a_v}]} |\W{k}{x'}|$.
\begin{lemma}\label{lem:sup_W}
Let $k$ be a positive integer. If $b > 2$, then we have
$$
\sup_{x' \in [0,b^{-a_v}]} |\W{k}{x'}| \leq
\frac{b^{-\mu(k)}}{\mb^{v-1}}
 \frac{b}{b-\Mb} \left(1-\left(\frac{\Mb}{b}\right)^v \right).
$$
\end{lemma}

\begin{proof}
We prove the lemma by induction on $v$.
If $v=1$, we have
\begin{align*}
\sup_{x' \in [0,b^{-a_1}]} |\W{k}{x'}| 
&=\sup_{x' \in [0,b^{-a_1}]} \left| \int_0^{x'} \W{0}{y}\, dy \right|\\
&= \sup_{x' \in [0,b^{-a_1}]} |x'|
= b^{-a_1}
= b^{-\mu(k)}.
\end{align*}
Hence the lemma holds for $v=1$.

Thus assume now that $v > 1$ and that the result holds for $v - 1$.
Let $x' \in [0, b^{-a_v}]$ be a real number
and $x' = c' b^{-a_{v-1}} + x''$
where $0 \leq c' < b^{-a_v + a_{v-1}}$ is an integer
and $0 \leq x'' < b^{-a_{v-1}}$ is a real number.
Then by Lemma~\ref{lem:I-value} \eqref{I-value3} we have
\begin{align}
|\W{k}{x'}|
&= \left|\int_0^{x'} \W{k'}{y}\, dy\right| \notag\\
&= \left|\sum_{i=0}^{c'-1} \int_0^{b^{-a_{v-1}}} \left((1-\omebar^{i\kappa_{v-1}})\I{k'} + \omebar^{i\kappa_{v-1}}\W{k'}{y}\right) dy \right. \notag\\
&\quad + \left. \int_0^{x''} \left((1-\omebar^{c'\kappa_{v-1}})\I{k'} + \omebar^{c'\kappa_{v-1}}\W{k'}{y}\right) dy \right| \notag\\
&\leq \left|b^{-a_{v-1}}\sum_{i=0}^{c'-1} (1-\omebar^{i\kappa_{v-1}})\I{k'} + x'' (1-\omebar^{c'\kappa_{v-1}})\I{k'}\right| \notag\\
&\quad
+ \left|\sum_{i=0}^{c'-1} \omebar^{i\kappa_{v-1}} \int_0^{b^{-a_{v-1}}} \W{k'}{y}\, dy + \omebar^{c'\kappa_{v-1}}\int_0^{x''} \W{k'}{y}\, dy\right|.\label{eq:sup-W-pf}
\end{align}

We estimate the supremum of the first term of \eqref{eq:sup-W-pf}.
Note that the first term of \eqref{eq:sup-W-pf} is equal to
$|b^{-a_{v-1}}\sum_{i=0}^{c'} (1-\omebar^{i\kappa_{v-1}})\I{k'}|$
if $x'' = b^{-a_{v-1}}$.
Using this and Lemma \ref{lem:convex-endpoint}, we have
\begin{align}
& \sup_{\substack{c' \in \bN, \, 0 \leq c' < b^{-a_v + a_{v-1}} \\ x'' \in [0,b^{-a_{v-1}}]}}\left|b^{-a_{v-1}}\sum_{i=0}^{c'-1} (1-\omebar^{i\kappa_{v-1}})\I{k'} + x'' (1-\omebar^{c'\kappa_{v-1}})\I{k'}\right| \notag \\
&\quad= \sup \max\left(\left|b^{-a_{v-1}}\sum_{i=0}^{c'-1} (1-\omebar^{i\kappa_{v-1}})\I{k'}\right|, \left|b^{-a_{v-1}}\sum_{i=0}^{c'} (1-\omebar^{i\kappa_{v-1}})\I{k'}\right|\right) \label{eq:lem:sup_W-sup-calculate1} \\
&\quad= \sup_{c' \in \bN, 0 \leq c' \leq b^{-a_v + a_{v-1}}} \left|b^{-a_{v-1}}\sum_{i=0}^{c'-1} (1-\omebar^{i\kappa_{v-1}})\I{k'}\right|, \label{eq:lem:sup_W-sup-calculate2}
\end{align}
where the supremum in \eqref{eq:lem:sup_W-sup-calculate1} is extended over all $c' \in \bN$ with $0 \leq c' < b^{-a_v + a_{v-1}}$.
By Lemmas \ref{lem:first-term2} and
\ref{lem:I-value}~\eqref{I-value1}, \eqref{eq:lem:sup_W-sup-calculate2} is bounded by
\begin{align*}
b^{-a_{v-1}} \frac{b^{-\mu(k')}}{\mb^{v-1}} b^{-a_v + a_{v-1}}
= \frac{b^{-\mu(k)}}{\mb^{v-1}}.
\end{align*}
Thus the supremum of the first term of \eqref{eq:sup-W-pf} is bounded by $b^{-\mu(k)}/{\mb^{v-1}}.$

We move on to the estimation of the supremum of the second term of \eqref{eq:sup-W-pf}.
We have
\begin{align*}
&\sup_{c', x''}\left|\sum_{i=0}^{c'-1} \omebar^{i\kappa_{v-1}} \int_0^{b^{-a_{v-1}}} \W{k'}{y}\, dy + \omebar^{c'\kappa_{v-1}}\int_0^{x''} \W{k'}{y}\, dy\right|\\
&= \sup_{c', x''}\left|\sum_{i=0}^{c'-1} \omebar^{i\kappa_{v-1}} \int_{x''}^{b^{-a_{v-1}}} \W{k'}{y}\, dy + \sum_{i=0}^{c'} \omebar^{i\kappa_{v-1}} \int_0^{x''} \W{k'}{y}\, dy\right|\\
&= \sup_{c', x''}\left|\frac{1 - \omebar^{c'\kappa_{v-1}}}{1 - \omebar^{\kappa_{v-1}}} \int_{x''}^{b^{-a_{v-1}}} \W{k'}{y}\, dy + \frac{1 - \omebar^{(c' + 1)\kappa_{v-1}}}{1- \omebar^{\kappa_{v-1}}}\int_0^{x''} \W{k'}{y}\, dy \right| \\
&\leq \sup_{x'' \in [0,b^{-a_{v-1}}]}\left|\frac{\Mb}{\mb} (b^{-a_{v-1}} - x'') + \frac{\Mb}{\mb}x'' \right| \cdot \sup_{y \in [0,b^{-a_{v-1}}]} |\W{k'}{y}|\\
&\leq \frac{\Mb}{\mb}b^{-a_{v-1}} \cdot \frac{b^{-\mu(k')}}{\mb^{v-2}}
 \frac{b}{b-\Mb} \left(1-\left(\frac{\Mb}{b}\right)^{v-1} \right)\\
&\leq \frac{b^{-\mu(k)}}{\mb^{v-1}}
 \frac{\Mb}{b-\Mb} \left(1-\left(\frac{\Mb}{b}\right)^{v-1} \right),
\end{align*}
where the supremums in the first, second and third lines are extended
over all $c' \in \bN$ with $0 \leq c' < b^{-a_v + a_{v-1}}$ and $x'' \in [0,b^{-a_{v-1}}]$,
and where we use the induction assumption for $v-1$ in the fourth inequality
and $b \cdot b^{-a_{v-1}} \leq b^{-a_v}$ in the last inequality.

By summing up the bounds obtained on each term of \eqref{eq:sup-W-pf}, we have
\begin{align*}
\sup_{x' \in [0,b^{-a_v}]} |\W{k}{x'}| 
&\leq \frac{b^{-\mu(k)}}{\mb^{v-1}} + \frac{b^{-\mu(k)}}{\mb^{v-1}}
 \frac{\Mb}{b-\Mb} \left(1-\left(\frac{\Mb}{b}\right)^{v-1} \right)\\
&= \frac{b^{-\mu(k)}}{\mb^{v-1}}
 \frac{b}{b-\Mb} \left(1-\left(\frac{\Mb}{b}\right)^v \right). \qedhere
\end{align*}
\end{proof}

Using the above lemma, we obtain an upper bound on $\Wnormof{k}{\infty}$.

\begin{proposition}\label{prop:Wnorm-infty}
Let $k \in \bN$.
If $b > 2$, we have
$$
\Wnormof{k}{\infty}
\leq \frac{b^{-\mu(k)}}{\mb^{v}} \left(\Mb + \Cv \right)^{\min(1, v)}.
$$
\end{proposition}
\begin{proof}
The case $k=0$ is obvious.
We assume that $k > 0$. 
Let $x \in [0,1)$ and
$x = cb^{-a_v} + x'$, where $0 \leq c < b^{a_v}$ is an integer
and $0 \leq x' < b^{-a_v}$ is a real number.
By Lemmas \ref{lem:I-value} and \ref{lem:sup_W}, we have
\begin{align*}
|\W{k}{x}|
&= \left|(1-\omebar^{c\kappa_v})\I{k}+ \omebar^{c\kappa_v} \W{k}{x'} \right|\\
&\leq \Mb|\I{k}| + \sup_{x' \in [0,b^{-a_v}]} |\W{k}{x'}|\\
&\leq \frac{b^{-\mu(k)}}{\mb^{v}} \left(\Mb + \Cv \right)^{\min(1, v)},
\end{align*}
which proves the proposition.
\end{proof}

\subsection{Dyadic case}\label{subsec:dyadic-W-bound}
In this subsection, we assume that $b=2$.
In the dyadic case, 
we can obtain the exact values of $\Wnormof{k}{1}$ and $\Wnormof{k}{\infty}$.
First we show properties of $\Wnorm{k}$ for the dyadic case.
\begin{lemma}\label{lem:W-dyadic}
Let $k \in \bN$. Assume that $b=2$ and $x_1, x_2 \in [0,1)$.
Then we have the following.
\begin{enuroman}
\item \label{W-dyadic-1}
Assume that $x_1 + x_2$ is a multiple of $2^{-a_v + 1}$.
Then we have
$\W{k}{x_1} = \W{k}{x_2}$.
\item \label{W-dyadic-2}
Assume that $x_1 + x_2$ is a multiple of $2^{-a_v}$ and
not a multiple of $2^{-a_v + 1}$.
If $k \neq 0$, then we have
$\W{k}{x_1} + \W{k}{x_2} = \W{k}{2^{-a_v}}$.
\item \label{W-dyadic-3}
The function $\W{k}{\cdot}$ is nonnegative.
\end{enuroman}
\end{lemma}

\begin{proof}
We prove the lemma by induction on $v$.
The results hold for $v = 0$ since $\W{0}{x} = 1$ for all $x \in [0,1)$.
Hence assume now that $v > 0$ and that the results hold for $v - 1$.

First we assume that $x_1 + x_2$ is a multiple of $2^{-a_v + 1}$.
Since $\Wnorm{k}$ has a period $2^{-a_v + 1}$ by Lemma~\ref{lem:Wperiod},
we can assume that $x_1, x_2 \in [0, 2^{-a_v + 1}]$.
Then we can assume that $x_1 \in [0, 2^{-a_v}]$ and that $x_2 = 2^{-a_v + 1} - x_1$.
Now we prove that $\W{k}{x_1} = \W{k}{x_2}$.
We have
\begin{equation}\label{eq:pf-of-dyadic-periodic-(i)}
\W{k}{x_2}
= \W{k}{2^{-a_v + 1}} - \int_{x_2}^{2^{-a_v + 1}} \overline{\wal_{2^{a_v-1}}(y)} \W{k'}{y}\, dy.
\end{equation}
The first term of the right hand side of \eqref{eq:pf-of-dyadic-periodic-(i)} is equal to zero
by Lemma \ref{lem:Wperiod}.
We now consider the second term.
We have $\wal_{2^{a_v-1}}(y) = -1$ for all $y \in [2^{-a_v}, 2^{-a_v + 1})$.
Further, by applying the induction assumption of \eqref{W-dyadic-1} for $v-1$ to $\Wnorm{k'}$, we have
$\W{k'}{y} = \W{k'}{2^{-a_v + 1} - y}$ for all $y \in [0, 2^{-a_v + 1}]$,
since $y + (2^{-a_v + 1} - y) = 2^{-a_v + 1}$ is a multiple of $2^{-a_{v-1} + 1}$.
Thus the second term of the right hand side of \eqref{eq:pf-of-dyadic-periodic-(i)} is equal to
\[
\int_{x_2}^{2^{-a_v + 1}} (-1) \W{k'}{2^{-a_v + 1} - y}\, dy
= -\int_{0}^{x_1} \W{k'}{y'}\, dy'
= -\W{k}{x_1},
\]
where $y' = 2^{-a_v + 1} - y$.
Thus the right hand side of \eqref{eq:pf-of-dyadic-periodic-(i)} is equal to $\W{k}{x_1}$,
which proves \eqref{W-dyadic-1} for $v$.

Second we assume that $x_1 + x_2$ is a multiple of $2^{-a_v}$
and not a multiple of $2^{-a_v + 1}$.
Similar to the first case, we can assume that 
$x_1, x_2 \in [0, 2^{-a_v}]$ and that $x_2 = 2^{-a_v} - x_1$.
By applying the induction assumption of \eqref{W-dyadic-1} for $v-1$ to $\Wnorm{k'}$, we have
$\W{k'}{y} = \W{k'}{2^{-a_v} - y}$ for all $y \in [0, 2^{-a_v}]$,
since $y + (2^{-a_v} - y) = 2^{-a_v}$ is a multiple of $2^{-a_{v-1} + 1}$.
Hence we have
\begin{align*}
\W{k}{x_1} + \W{k}{x_2}
&= \int_0^{x_1} \W{k'}{y} \, dy + \int_0^{x_2} \W{k'}{y}\, dy\\
&= \int_0^{x_1} \W{k'}{y} \, dy + \int_{0}^{x_2} \W{k'}{2^{-a_v} - y}\, dy\\
&= \int_0^{x_1} \W{k'}{y} \, dy + \int_{2^{-a_v} - x_2}^{2^{-a_v}} \W{k'}{y}\, dy\\
&= \int_0^{2^{-a_v}} \W{k'}{y}\, dy\\
&= \W{k}{2^{-a_v}},
\end{align*}
which proves \eqref{W-dyadic-2} for $v$.

Finally we prove that $\W{k}{x}$ is nonnegative for all $x \in [0,1)$.
By the induction assumption of \eqref{W-dyadic-3} for $v-1$, $\W{k'}{x}$ is nonnegative for $x \in [0,1)$.
For $x \in [0, 2^{-a_v}]$, we have
$
\W{k}{x}
= \int_0^{x} \W{k'}{y}dy,
$
and thus $\W{k}{x}$ is nonnegative for $x \in [0, 2^{-a_v}]$. 
Hence by \eqref{W-dyadic-1} for $v$
and Lemma~\ref{lem:Wperiod},
$\W{k}{x}$ is nonnegative for $x \in [0,1)$.
\end{proof}

Now we are ready to consider $\Wnormof{k}{q}$ for $1 \leq q \leq \infty$.

First we consider $\Wnormof{k}{1}$.
By Lemmas \ref{lem:I-value}~\eqref{I-value1} and \ref{lem:W-dyadic}~\eqref{W-dyadic-3},
we have
\begin{align*}
\Wnormof{k}{1}
= \int_0^1 |\W{k}{x}| \, dx 
= \int_0^1 \W{k}{x} \, dx
= 2^{-\mu(k)-v}.
\end{align*}

Second we consider $\Wnormof{k}{\infty}$.
If $k=0$, we have $\Wnormof{k}{\infty} = 1$.
We assume that $k > 0$.
Considering the symmetry and the non-negativity of $\Wnorm{k}$
given by Lemma~\ref{lem:W-dyadic},
we have
\begin{align*}
\Wnormof{k}{\infty}
&= \sup_{x \in [0, 2^{-a_v}]} |\W{k}{x}| \, dx\\
&= \sup_{x \in [0, 2^{-a_v}]} \left| \int_{0}^{x} \W{k'}{y}\, dy \right|\\
&= \int_{0}^{2^{-a_v}} \W{k'}{y}\, dy\\
&= \W{k}{2^{-a_v}} = 2^{-\mu(k)-v+1}.
\end{align*}
Thus we have $\Wnormof{k}{\infty} = 2^{-\mu(k)-v+ \min(1,v)}$ for all $k \in \bN$.

Finally we consider $\Wnormof{k}{q}$.
By H\"{o}lder's inequality, we have 
\begin{align*}
\Wnormof{k}{q}
&=\left(\int_{[0,1)^s}|\W{k}{x}|\cdot |\W{k}{x}|^{q-1} \, dx\right)^{1/q}\\
&\leq (\Wnormof{k}{1} \Wnormof{k}{\infty}^{q-1})^{1/q}\\
&\leq 2^{-\mu(k)-v+(1-1/q)\min(1,v)}.
\end{align*}

We have shown the following proposition.
\begin{proposition}\label{prop:Wnorm-dyadic}
Let $b=2$.
For $k \in \bN$ and $1 \leq q \leq \infty$, we have
$$
\Wnormof{k}{q} \leq 2^{-\mu(k)-v+(1-1/q)\min(1,v)},
$$
and equality holds if $q=1$ or $q=\infty$.
\end{proposition}

\subsection{Bounds on the Walsh coefficients of smooth functions}
\label{subsec:bound-smooth}
For a positive integer $\alpha$ and $k \in \bN$, we define 
\begin{equation}
\mu_\alpha(k)
:= \mu(\sumless{\alpha})
= \begin{cases}
0 & \text{for $k=0$}, \\
a_1 + \dots + a_v & \text{for $1 \leq v \leq \alpha$},\\
a_1 + \dots + a_\alpha & \text{for $v \geq \alpha$},
\end{cases}
\end{equation}
as in \cite{Dick2009dwc}.
By \eqref{eq:Walsh-bound-smooth}, Proposition~\ref{prop:Wnorm-infty} and
Proposition~\ref{prop:Wnorm-dyadic},
we obtain the following bound on the Walsh coefficients of smooth functions.
\begin{theorem}\label{thm:Walsh-bound-normal}
Let $f \in C^\alpha[0,1]$ and  $k \in \bN$.
If $b>2$, we have
\begin{align*}
|\widehat{f}(k)|
&\leq \|f^{(\min(\alpha ,v))}\|_{L^1}
\frac{b^{-\mu_\alpha(k)}}{\mb^{\min(\alpha ,v)}} \times\\
&\qquad \left(\Mb + \Cv[\min(\alpha,v)]\right)^{\min(1,v)}.
\end{align*}
If $b=2$, for $1 \leq p \leq \infty$ we have
\begin{align*}
|\widehat{f}(k)|
\leq\|f^{(\min(\alpha ,v))}\|_{L^p} \cdot 2^{-\mu_{\alpha}(k)- \min(\alpha, v) +\min(1,v)/p}.
\end{align*}
\end{theorem}

The $s$-variate case follows in the same way as the univariate case.
\begin{theorem}\label{thm:Walsh-bound-normal-s}
Let $\bsk = (k_1, \dots, k_s) \in \bN^s$.
Assume that $f \colon [0,1]^s \to \bR$
has continuous mixed partial derivatives up to order $\alpha_j$ in each variable $x_j$.
Let $n_j := \min(\alpha_j, v(k_j))$ for $1 \leq j \leq s$.
Then, if $b>2$, we have
\begin{align*}
|\widehat{f}(\bsk)|
&\leq \|f^{(n_1, \dots, n_s)}\|_{L^1}
\prod_{j=1}^s \frac{b^{-\mu_{\alpha_j}(k_j)}}{\mb^{n_j}} \times \\
&\qquad \left(\Mb + \Cv[n_j]\right)^{\min(1,v(k_j))}.
\end{align*}
If $b=2$, for $1 \leq p \leq \infty$ we have
\begin{align*}
|\widehat{f}(\bsk)|
\leq\|f^{(n_1, \dots, n_s)}\|_{L^p}\times
\prod_{j=1}^s 2^{-\mu_{\alpha_j}(k_j)- n_j +\min(1,v(k_j))/p}.
\end{align*}
\end{theorem}
As a corollary, we give a sufficient condition for an infinitely differentiable function
that its Walsh coefficients decay with order $O(b^{-\mu(\bsk)})$.
\begin{corollary}\label{cor:Walsh-bound-infty}
Let $f \in C^\infty[0,1]^s$ and $r_j > 0$ be positive real numbers for $1 \leq j \leq s$.
Assume that
there exists a positive real number $D$ such that
$$
\|f^{(n_1, \dots, n_s)}\|_{L^1} 
\leq D \prod_{j=1}^s r_j^{n_j}
$$
holds for all $n_1, \dots, n_s \in \bN$.
Then for all $\bsk \in \bN^s$ we have
$$
|\widehat{f}(\bsk)|
\leq D b^{-\mu(\bsk)} \prod_{j=1}^s (r_j \mb^{-1})^{v(k_j)} C_b^{\min(1, v(k_j))},
$$
where $C_b$ is a constant defined as
$$
C_b
= \begin{cases}
2 & \text{for $b = 2$}, \\
\displaystyle \Mb + \frac{b\mb}{b-\Mb} & \text{for $b \neq 2$}.
\end{cases}
$$
In particular, if $r_j = \mb$ holds for all $1 \leq j \leq s$, then
$
|\widehat{f}(\bsk)|
\in O(b^{-\mu(\bsk)})
$
holds.
\end{corollary}

\section{Another formula for the Walsh coefficients}\label{sec:Another-formula}
In this section, we give another formula for the Walsh coefficients.
For this purpose, we introduce functions $\Wextranorm{j}{k}$
and their integration values $\Iextra{j}{k}$ for $j, k \in \bN$.

\begin{definition}
For $j, k \in \bN$, we define functions
$\Wextranorm{j}{k} \colon [0,1] \to \bC$
and complex numbers $\Iextra{j}{k}$ recursively as
\begin{align*}
\Wextra{0}{k}{x} &:= \W{k}{x},\\
\Iextra{j}{k} &:= \int_0^1 \Wextra{j}{k}{x} \, dx ,\\
\Wextra{j+1}{k}{x} &:= \int_0^x (\Wextra{j}{k}{y} - \Iextra{j}{k})\, dy.
\end{align*}
\end{definition}
We note that $\Wextra{j}{k}{0} = \Wextra{j}{k}{1}  = 0$ for all $j, k \in \bN$
with $(j,k) \neq (0,0)$.

We now establish another formula for the Walsh coefficients of smooth functions.
\begin{theorem}\label{thm:extra_exp}
Let $k, r \in \bN$
and $f \in C^{v+r}[0,1]$.
Then we have
\begin{align*}
\widehat{f}(k)
&= \sum_{i=0}^{r} (-1)^{v+i} \Iextra{i}{k} \int_0^1 f^{(v+i)}(x) \, dx \notag \\
&\qquad+ (-1)^{v+r} \int_0^1 f^{(v+r)}(x) (\Wextra{r}{k}{x} - \Iextra{r}{k}) \, dx.
\end{align*}
\end{theorem}

\begin{proof}
We prove the theorem by induction on $r$.
We have already proved the case $r=0$ in Theorem~\ref{thm:Walsh_formula}.
Thus assume now that $r \geq 1$ and that the result holds for $r - 1$.
By the induction assumption for $r-1$, we have
\begin{align*}
\widehat{f}(k)
&= \sum_{i=0}^{r-1} (-1)^{v+i} \Iextra{i}{k} \int_0^1 f^{(v+i)}(x) \, dx\\
&\qquad + (-1)^{v+r-1} \int_0^1 f^{(v+r-1)}(x) (\Wextra{r-1}{k}{x} - \Iextra{r-1}{k}) \, dx\\
&= \sum_{i=0}^{r-1} (-1)^{v+i} \Iextra{i}{k} \int_0^1 f^{(v+i)}(x) \, dx\\
&\qquad + (-1)^{v+r-1} \left([f^{(v+r-1)}(x) \Wextra{r}{k}{x}]_0^1 - \int_0^1 f^{(v+r)}(x) \Wextra{r}{k}{x} \, dx \right)\\
&= \sum_{i=0}^{r} (-1)^{v+i} \Iextra{i}{k} \int_0^1 f^{(v+i)}(x) \, dx\\
&\qquad + (-1)^{v+r} \int_0^1 f^{(v+r)}(x) (\Wextra{r}{k}{x} - \Iextra{r}{k}) \, dx,
\end{align*}
where we use
$\Wextra{r}{k}{0} = \Wextra{r}{k}{1} = 0$
in the third equality.
This proves the result for $r$.
\end{proof}

\section{The Walsh coefficients of Bernoulli polynomials}\label{sec:coeff-of-Bernoulli}
In this section, we analyze the decay of the Walsh coefficients of Bernoulli polynomials.

For $r \geq 0$, we denote $B_r(\cdot)$ the Bernoulli polynomial of degree $r$
and $b_r(x) = B_r(x)/r!$.
For example, we have 
$B_0(x)=1$, $B_1(x)=x-1/2$, $B_2(x)=x^2-x+1/6$ and so on.
Those polynomials have the following properties:
For all $r \geq 1$ we have
\begin{equation}\label{eq:b_r-property}
b'_r(x) = b_{r-1}(x) \quad \text{and} \quad \int_0^1 b_r(x) \, dx = 0,
\end{equation}
and for all $r \in \bN$ we have
\begin{equation}\label{eq:b_r-symmetry}
b_r(1-x) = (-1)^r b_r(x),
\end{equation}
see \cite[Chapter~23]{Abramowitz1992hmf}.
We clearly have $b'_0(x) = 0$ and $\int_0^1 b_0(x)=1$.

The Walsh coefficients of Bernoulli polynomials are given as follows.
If $r < v$, then by Theorem~\ref{thm:Walsh_formula} and \eqref{eq:b_r-property}
we have
$$
\widehat{b_r}(k) = (-1)^v \int_0^1 b_r^{(v)}(x) \W{k}{x} \, dx = 0.
$$
If $r \geq v$, then by Theorem~\ref{thm:extra_exp} and \eqref{eq:b_r-property} we have
\begin{align*}
\widehat{b_r}(k)
&= \sum_{i=0}^{r-v} (-1)^{v+i} \Iextra{i}{k} \int_0^1 b_r^{(v+i)}(x) \, dx\\
&\quad+ (-1)^{r} \int_0^1 {b_r}^{(r)}(x) (\Wextra{r-v}{k}{x} - \Iextra{r-v}{k}) \, dx\\
&= (-1)^{r} \Iextra{r-v}{k}.
\end{align*}
Now we proved:
\begin{lemma}\label{lem:W(b_r)=I}
For positive integers $k$ and $r$, we have
\begin{align*}
\widehat{b_r}(k)
=&
\begin{cases}
0 & \text{if $r<v$},\\
(-1)^{r} \Iextra{r-v}{k} & \text{if $r \geq v$}.
\end{cases}
\end{align*}
\end{lemma}

In the following, we give upper bounds on
$\|\Wextranorm{j}{k} - \Iextra{j}{k}\|_{L^\infty}$,
$|\Iextra{j}{k}|$ and $\Wextranormof{j}{k}{\infty}$,
which give bounds on the Walsh coefficients of Bernoulli polynomials and smooth functions.
First we compute $\Wextranorm{j}{k}$ and $\Iextra{j}{k}$.

\begin{lemma}\label{lem:Wextra_period}
Let $k, j \in \bN$.
Let $x \in [0,1)$ and $x = c b^{-a_v} + x'$
with $c \in \bN$ and $x' \in [0, b^{-a_v})$.
Then we have
\begin{enuroman}
\item\label{lem:Wextra_period-1}
$\displaystyle
\Wextra{j}{k}{x}
= \frac{1-\omebar^{c\kappa_v}}{1-\omebar^{\kappa_v}} \Wextra{j}{k}{b^{-a_v}}
+ \omebar^{c\kappa_v} \Wextra{j}{k}{x'}$,
\item\label{lem:Wextra_period-2}
$\displaystyle
\Iextra{j}{k} = \frac{\Wextra{j}{k}{b^{-a_v}}} {1-\omebar^{\kappa_v}}$.
\end{enuroman}
\end{lemma}

\begin{proof}
We prove the lemma by induction on $j$.
We have already proved the case $j=0$ in Lemmas \ref{lem:Wperiod} and \ref{lem:I-value}.
Thus assume now that $j \geq 1$ and that the result holds for $j - 1$.
Then we have
\begin{align*}
\Wextra{j}{k}{x}
&= \int_0^x (\Wextra{j-1}{k}{y} - \Iextra{j-1}{k})\, dy\\
&= \sum_{i=0}^{c-1} \int_{0}^{b^{-a_v}} \left(\frac{-\omebar^{i \kappa_v}}{1-\omebar^{\kappa_v}} \Wextra{j-1}{k}{b^{-a_v}} + \omebar^{i \kappa_v}\Wextra{j-1}{k}{y}\right) dy\\
&\qquad + \int_{0}^{x'} \left(\frac{-\omebar^{c \kappa_v}}{1-\omebar^{\kappa_v}} \Wextra{j-1}{k}{b^{-a_v}} + \omebar^{c \kappa_v}\Wextra{j-1}{k}{y}\right) dy \\
&= \sum_{i=0}^{c-1} \omebar^{i \kappa_v} \Wextra{j}{k}{b^{-a_v}}
+ \omebar^{c \kappa_v} \Wextra{j}{k}{x'}\\
&= \frac{1-\omebar^{c\kappa_v}}{1-\omebar^{\kappa_v}} \Wextra{j}{k}{b^{-a_v}}
+ \omebar^{c\kappa_v} \Wextra{j}{k}{x'},
\end{align*}
where we use the induction assumption for $j-1$ in the second and third equalities
and the definition of $\Wextranorm{j}{k}$ in the third equality.
This proves \eqref{lem:Wextra_period-1} for $j$.

Now we compute $\Iextra{j}{k}$.
Replacing $\Wnorm{k}$ by $\Wextranorm{j}{k}$ in \eqref{eq:I-to-W}, we have
$
\Iextra{j}{k} = \Wextra{j}{k}{b^{-a_v}}/(1-\omebar^{\kappa_v}),
$
which proves \eqref{lem:Wextra_period-2} for $j$.
\end{proof}

The following lemmas give bounds on $\|\Wextranorm{j}{k} - \Iextra{j}{k}\|_{L^\infty}$,
$|\Iextra{j}{k}|$ and $\Wextranormof{j}{k}{\infty}$ for the non-dyadic case.

\begin{lemma}\label{lem:Wextra_sup_av}
Let $j \in \bN$.
If $b \neq 2$, for any positive integer $k$ we have
$$
\|\Wextranorm{j}{k} - \Iextra{j}{k}\|_{L^\infty}
\leq \frac{b^{-\mu(k) - j a_v}}{\mb^{v+j}} \left(1 + \Cv\right).
$$
\end{lemma}

\begin{proof}
Let $x\in[0,1)$ and $x = c b^{-a_v} + x'$ with $c \in \bN$ and $x' \in [0, b^{-a_v})$.
First assume that $j=0$.
Then it follows from Lemmas \ref{lem:I-value} and \ref{lem:sup_W} that
\begin{align*}
|\Wextra{0}{k}{x} - \Iextra{0}{k}|
&= \left|-\omebar^{c\kappa_v}\I{k} + \omebar^{c\kappa_v}\W{k}{x'} \right|\\
&\leq |\I{k}| + \sup_{x' \in [0,b^{-a_v}]} |\W{k}{x'}|\\
&\leq \frac{b^{-\mu(k)}}{\mb^{v}} \left(1 + \Cv \right),
\end{align*}
which proves the case $j=0$.

Now we assume that $j>0$.
Then it follows from Lemma~\ref{lem:Wextra_period} that
\begin{align*}
|\Wextra{j}{k}{x} - \Iextra{j}{k}|
&= \left|\frac{1-\omebar^{c\kappa_v}}{1-\omebar^{\kappa_v}} \Wextra{j}{k}{b^{-a_v}}
+ \omebar^{c\kappa_v} \Wextra{j}{k}{x'} - \frac{\Wextra{j}{k}{b^{-a_v}}} {1-\omebar^{\kappa_v}} \right|\\
&= \left|\frac{-1}{1-\omebar^{\kappa_v}} \Wextra{j}{k}{b^{-a_v}} + \Wextra{j}{k}{x'}\right|\\
&= \left|\frac{-1}{1-\omebar^{\kappa_v}} (\Wextra{j}{k}{b^{-a_v}} - \Wextra{j}{k}{x'}) - \frac{\omebar^{\kappa_v}}{1-\omebar^{\kappa_v}}\Wextra{j}{k}{x'} \right|\\
&\leq \frac{1}{\mb} \left|\int_{x'}^{b^{-a_v}} (\Wextra{j-1}{k}{y} - \Iextra{j-1}{k})\, dy\right.\\
&\qquad\qquad \left. + \omebar^{\kappa_v} \int_{0}^{x'} (\Wextra{j-1}{k}{y} - \Iextra{j-1}{k})\, dy \right|\\
&\leq \frac{1}{\mb} (b^{-a_v}-x') \sup_{y \in [0, b^{-a_v}]} |\Wextra{j-1}{k}{y} - \Iextra{j-1}{k}|\\
&\quad+ \frac{1}{\mb} x' \sup_{y \in [0, b^{-a_v}]} |\Wextra{j-1}{k}{y} - \Iextra{j-1}{k}|\\
&\leq \frac{b^{-a_v}}{\mb} \|\Wextranorm{j-1}{k} - \Iextra{j-1}{k}\|_{L^\infty}.
\end{align*}
Using the case $j=0$ and this evaluation inductively,
we have the case $j>0$.
\end{proof}

\begin{lemma}\label{lem:Iextra_sup}
Let $j$ and $k$ be positive integers.
If $b>2$, then we have
$$
|\Iextra{j}{k}|
\leq \frac{b^{-\mu(k) - j a_v}}{\mb^{v+j}} \left(1 + \Cv\right).
$$
\end{lemma}

\begin{proof}
By Lemmas~\ref{lem:Wextra_period} and \ref{lem:Wextra_sup_av}, we have
\begin{align*}
|\Iextra{j}{k}|
&= |\Wextra{j}{k}{b^{-a_v}}/(1 - \omebar^{\kappa_v})|\\
&\leq \frac{1}{\mb} \int_0^{b^{-a_v}} |\Wextra{j-1}{k}{y} - \Iextra{j-1}{k}|\, dy\\
&\leq \frac{b^{-a_v}}{\mb} \|\Wextra{j-1}{k}{y} - \Iextra{j-1}{k}\|_{L^\infty}\\
&\leq \frac{b^{-\mu(k) - j a_v}}{\mb^{v+j}} \left(1 + \Cv\right). \qedhere
\end{align*}
\end{proof}

\begin{lemma}\label{lem:Wextra_sup}
Let $j$ and $k$ be positive integers.
If $b>2$, then we have
$$
\Wextranormof{j}{k}{\infty}
\leq \frac{b^{-\mu(k) - j a_v}}{\mb^{v+j}} \Mb \left(1 + \Cv \right).
$$
\end{lemma}

\begin{proof}
Let $x \in [0,1)$ and
$x = cb^{-a_v} + x'$, where $0 \leq c < b^{a_v}$ is an integer
and $0 \leq x' < b^{-a_v}$ is a real number.
Then we have
\begin{align*}
\Wextra{j}{k}{x}
&= \frac{1-\omebar^{c\kappa_v}}{1-\omebar^{\kappa_v}} \Wextra{j}{k}{b^{-a_v}}
+ \omebar^{c\kappa_v} \Wextra{j}{k}{x'}\\
&= \frac{1-\omebar^{c\kappa_v}}{1-\omebar^{\kappa_v}} (\Wextra{j}{k}{b^{-a_v}} - \Wextra{j}{k}{x'})
+ \frac{1-\omebar^{(c+1)\kappa_v}}{1-\omebar^{\kappa_v}} \Wextra{j}{k}{x'}\\
&= \frac{1 - \omebar^{c \kappa_v}}{1 - \omebar^{\kappa_v}} \int_{x'}^{b^{-a_v}} (\Wextra{j-1}{k}{y} - \Iextra{j-1}{k})\, dy\\
&\qquad + \frac{1 - \omebar^{(c+1) \kappa_v}}{1 - \omebar^{\kappa_v}} \int_{0}^{x'} \ (\Wextra{j-1}{k}{y} - \Iextra{j-1}{k})\, dy.
\end{align*}
Thus we have
\begin{align*}
|\Wextra{j}{k}{x}|
&\leq \left|\frac{1 - \omebar^{c \kappa_v}}{1 - \omebar^{\kappa_v}} \int_{x'}^{b^{-a_v}} (\Wextra{j-1}{k}{y} - \Iextra{j-1}{k})\, dy \right| \\
&\qquad + \left|\frac{1 - \omebar^{(c+1) \kappa_v}}{1 - \omebar^{\kappa_v}} \int_{0}^{x'} (\Wextra{j-1}{k}{y} - \Iextra{j-1}{k})\, dy\right|\\
&\leq \frac{\Mb}{\mb}b^{-a_v} \|\Wextranorm{j-1}{k} - \Iextra{j-1}{k}\|_{L^\infty}\\
&\leq \frac{b^{-\mu(k) - j a_v}}{\mb^{v+j}} \Mb \left(1 + \Cv\right). \qedhere
\end{align*}
\end{proof}

We also consider the dyadic case.
\begin{lemma}\label{lem:extra_sup-dyadic}
Let $k$ be a positive integer and $j \in \bN$.
If $b=2$, then we have the following.
\begin{enuroman}
\item\label{W-I_sup-dyadic}
$\|\Wextra{j}{k}{x}-\Iextra{j}{k}\|_{L^\infty}\le 2^{-j(a_v+1)-\mu (k)-v}$,
\item\label{I_sup-dyadic}
$|\Iextra{j}{k}| \leq 2^{-j(a_v+1)-\mu (k)-v}$,
\item\label{W_sup-dyadic}
$\Wextranormof{j}{k}{\infty} \leq 2^{-j(a_v+1)-\mu (k)-v+1}$.
\item If $j$ is odd, then $\Iextra{j}{k}=0$.
\end{enuroman}
\end{lemma}

\begin{proof}
Lemma~\ref{lem:I-value} and Proposition~\ref{prop:Wnorm-dyadic} imply 
\eqref{I_sup-dyadic} and \eqref{W_sup-dyadic} for $j=0$.

Since $\Wextra{0}{k}{x}$ and $\Iextra{0}{k}$ are nonnegative, we have 
\begin{align*}
\|\Wextra{0}{k}{x}-\Iextra{0}{k}\|_{L^\infty}
&\le\max \left( |\Wextranormof{0}{k}{\infty}-\Iextra{0}{k}|,|0-\Iextra{0}{k}|\right)\\
&\le 2^{-\mu(k)-v},
\end{align*}
and thus \eqref{W-I_sup-dyadic} for $j=0$ holds.

For the proof for the case $j > 0$, we note that
parts of the proofs of Lemmas \ref{lem:Wextra_sup_av},
\ref{lem:Iextra_sup} and \ref{lem:Wextra_sup} are valid even in the dyadic case:
For $b=2$ we have
\begin{align*}
|\Wextra{j}{k}{x} - \Iextra{j}{k}|
&\leq \frac{b^{-a_v}}{\mb} \|\Wextranorm{j-1}{k} - \Iextra{j-1}{k}\|_{L^\infty},\\
|\Iextra{j}{k}|
&\leq \frac{b^{-a_v}}{\mb} \|\Wextra{j-1}{k}{y} - \Iextra{j-1}{k}\|_{L^\infty},\\
|\Wextra{j}{k}{x}|
&\leq \frac{\Mb}{\mb}b^{-a_v} \|\Wextranorm{j-1}{k} - \Iextra{j-1}{k}\|_{L^\infty}.
\end{align*}
Combining these inequalities and the case $j=0$,
we have \eqref{W-I_sup-dyadic}, \eqref{I_sup-dyadic} and \eqref{W_sup-dyadic}
for $j>0$.

Now we assume that $j$ is odd and
prove $\Iextra{j}{k} = 0$.
By Lemma~\ref{lem:W(b_r)=I}, we have
\begin{align*}
\widehat{b_{v+j}}(k)=(-1)^{v+j}\Iextra{j}{k}.
\end{align*}
Hence it suffices to show $\widehat{b_{v+j}}(k) = 0$.
Since $j$ is odd, by \eqref{eq:b_r-symmetry}
we have $b_{v+j}(x) = (-1)^{v+1}b_{v+j}(1-x)$.
Furthermore,
$\wal_{k}(x)=(-1)^v\wal_{k}(1-x)$ holds for all but finitely many
$ x \in [0,1)$,
since we have
$\wal_{2^{a_i-1}}(x) = -\wal_{2^{a_i-1}}(1-x)$
for $x\in [0,1)\backslash \{l/2^{a_i} \mid 0 \leq l < 2^{a_i}\}$
and
$\wal_{k}(x) = \prod_{i=1}^v \wal_{2^{a_i-1}}(x)$.
Hence we have
\begin{align*}
\widehat{b_{v+j}}(k)
&=\int_0^{\frac{1}{2}} b_{v+j}(x)\wal_{k}(x)\, dx + \int_{\frac{1}{2}}^1b_{v+j}(x)\wal_{k}(x)\, dx\\
&=\int_0^{\frac{1}{2}} b_{v+j}(x)\wal_{k}(x)\, dx + \int_0^{\frac{1}{2}}b_{v+j}(1-x)\wal_{k}(1-x)\, dx\\
&=\int_0^{\frac{1}{2}} b_{v+j}(x)\wal_{k}(x)\, dx - \int_0^{\frac{1}{2}}b_{v+j}(x)\wal_{k}(x)\, dx\\
&=0. \qedhere
\end{align*}
\end{proof}

Now we are ready to analyze the decay of the Walsh coefficients of Bernoulli polynomials.
For a positive integer $\alpha$ and $k \in \bN$, we define 
\begin{equation}\label{eq:def-muper}
\muper{\alpha}{k}
= \begin{cases}
0 & \text{for $k=0$}, \\
a_1 + \dots + a_v + (\alpha-v)a_v & \text{for $1 \leq v \leq \alpha$},\\
a_1 + \dots + a_\alpha & \text{for $v \geq \alpha$},
\end{cases}
\end{equation}
as in \cite{Dick2009dwc}.
By Lemmas~\ref{lem:W(b_r)=I}, \ref{lem:Iextra_sup}
and \ref{lem:extra_sup-dyadic},
we have the following bound on the Walsh coefficients of Bernoulli polynomials.

\begin{theorem}\label{thm;b_r(k)}
For positive integers $k$ and $r$, we have
\begin{align*}
|\widehat{b_r}(k)|
\begin{cases}
=0 & \text{if $r<v$},\\
=0 & \text{if $r \geq v$, $r-v$ is odd and $b=2$},\\
\displaystyle \leq 2^{-\muper{r}{k}-r}
& \text{if $r \geq v$, $r-v$ is even and $b=2$},\\
\displaystyle \leq \frac{b^{-\muper{r}{k}}}{\mb^{r}} c_{b,v}
& \text{if $r \geq v$ and $b \neq 2$},
\end{cases}
\end{align*}
where $\displaystyle c_{b,v} := 1 + \Cv$.
\end{theorem}

\section{The Walsh coefficients of functions in Sobolev spaces}\label{sec:Sobolev}
In this section, we consider functions in the Sobolev space
\begin{align*}
\Sob := \{f \colon [0,1] \to \bR &\mid
\text{$f^{(i)}$: absolutely continuous for $i=0, \dots, \alpha-1$,}\\
&\qquad \qquad  f^{(\alpha)} \in L^2[0,1]\}
\end{align*}
for which $\alpha \geq 1$
as in \cite{Dick2009dwc}.
The inner product is given by
$$
\innerprod{f}{g}
= \sum_{i=0}^{\alpha-1} \int_0^1 f^{(i)}(x) \, dx \int_0^1 g^{(i)}(x) \, dx
	+ \int_0^1 f^{(\alpha)}(x) g^{(\alpha)}(x) \, dx.
$$
and the corresponding norm in $\Sob$ is given by
$\Sobnorm{f} := \sqrt{\innerprod{f}{f}}$.
The space $\Sob$ is a reproducing kernel Hilbert space
(see \cite{Aronszajn1950trk} for general information on reproducing kernel Hilbert spaces).
The reproducing kernel for this space is given by
$$
\mathcal{K}(x,y) = \sum_{i=0}^\alpha b_i(x) b_i(y) - (-1)^\alpha \widetilde{b}_{2\alpha}(x-y),
$$
where
\begin{equation*}
\widetilde{b}_{\alpha}(x-y)
:=
\begin{cases}
b_{\alpha}(|x-y|) & \text{if $\alpha$ is even},\\
(-1)^{1_{x<y}}b_{\alpha}(|x-y|) & \text{if $\alpha$ is odd},
\end{cases}
\end{equation*}
where we define $1_{x<y}$ to be $1$ for $x<y$ and $0$ otherwise,
see 
\cite[Lemma~2.1]{Craven1978/79snd}.
We have
\begin{align}\label{eq:Bernoulli_exp}
f(y)
&=\innerprod{f}{\mathcal{K}(\cdot,y)} \notag \\
&=\sum_{i=0}^{\alpha}\int_0^1 f^{(i)}(x)\, dx \, b_i(y) -(-1)^{\alpha}\int_0^1 f^{(\alpha)}(x) \widetilde{b}_{\alpha}(x-y)\, dx,
\end{align}
see the proof of \cite[Lemma~2.1]{Craven1978/79snd}. This implies that
\begin{align}
\label{eq:WB_formula}
\widehat{f}(k)
=\sum_{i=0}^{\alpha}\int_0^1 f^{(i)}(x)\, dx \, \widehat{b_i}(k) - (-1)^{\alpha}\int_0^1 f^{(\alpha)}(x) \int_0^1\widetilde{b}_{\alpha}(x-y)\overline{\wal_k(y)}\, dy \, dx.
\end{align}
However, we have already proved two formulas for the Walsh coefficients:
For $f \in C^\alpha[0,1]$,
in the case $\alpha \geq v$
we have Theorem~\ref{thm:extra_exp} for $r = \alpha-v$,
which is written as
\begin{align}\label{eq:exp_a>v}
\widehat{f}(k)
&= \sum_{i=v}^{\alpha} (-1)^{i} \Iextra{i-v}{k} \int_0^1 f^{(i)}(x) \, dx \notag \\
&\quad+ (-1)^{\alpha} \int_0^1 f^{(\alpha)}(x) (\Wextra{\alpha-v}{k}{x} - \Iextra{\alpha-v}{k}) \, dx,
\end{align}
and in the case $\alpha < v$ we have Theorem~\ref{thm:Walsh_formula} for $n=\alpha$,
which is written as
\begin{align}\label{eq:exp_a<v}
\widehat{f}(k)
= (-1)^\alpha \int_0^1 f^{(\alpha)}(x)
\overline{\wal_{\summore{\alpha}}(x)} \W{\sumless{\alpha}}{x} \, dx.
\end{align}
In this section, we show that Formulas \eqref{eq:exp_a>v} and \eqref{eq:exp_a<v}
are also valid for $f \in \Sob$
and give an upper bound for the Walsh coefficients of functions in $\Sob$.

\subsection{Formula for the Walsh coefficients of functions in Sobolev spaces}
First we consider the case $\alpha \geq v$.
The following lemma is needed to show that \eqref{eq:exp_a>v} is also valid
for $f \in \Sob$.
\begin{lemma}\label{lem:function_a>v}
Assume $\alpha \geq v$.
Define functions $h_1, h_2 \colon [0,1] \to \bC$ as
\begin{align*}
h_1(x) &:= -\int_0^1\widetilde{b}_{\alpha}(x-y)\overline{\wal_k(y)} \, dy ,\\
h_2(x) &:= \Wextra{\alpha-v}{k}{x} - \Iextra{\alpha-v}{k}.
\end{align*}
Then $h_1(x) = h_2(x)$ holds for all $x \in [0,1]$.
\end{lemma}

\begin{proof}
For $f \in C^\alpha[0,1]$ both formulas
\eqref{eq:WB_formula} and \eqref{eq:exp_a>v} hold.
Furthermore, by Lemma~\ref{lem:W(b_r)=I}, the first term of each formula is equal.
Hence we have
\begin{align*}
\int_0^1 f^{(\alpha)}(x) h_1(x) \, dx
= \int_0^1 f^{(\alpha)}(x) h_2(x) \, dx
\end{align*}
for all $f \in C^\alpha[0,1]$.
It is well known that if $h \colon [0,1] \to \bC$ is continuous and
$\int_0^1 g(x)h(x) \, dx = 0$ holds for all continuous functions $g \in C^0[0,1]$,
then $h(x) = 0$ holds.
Thus it suffices to show that $h_1$ and $h_2$ 
are continuous.

By definition, $h_2$ is continuous.
Now we prove that $h_1$ is continuous.
Fix $\epsilon > 0$.
Since $b_{\alpha}(z)$ is uniformly continuous on $z \in[0,1]$,
there exists $\delta_1$ such that
$|b_{\alpha}(z)-b_{\alpha}(z')|<\epsilon/2$
for all $z, z' \in [0,1]$ with $|z-z'|<\delta_1$.
Let
$
\delta_2=\min\left( 4^{-1} \epsilon (\max_{z\in [0,1]} |b_{\alpha}(z)|)^{-1}, \delta_1\right).
$
We fix $x \in [0,1]$ and prove $|h_1(x) - h_1(x')| \leq \epsilon$
for all $x' \in [0,1]$ with $|x-x'|<\delta_2$.
Without loss of generality, we can assume that $x < x'$.
Then we have
\begin{align*}
&\left|\int_0^1\widetilde{b}_{\alpha}(x-y)\overline{\wal_k(y)}\, dy-\int_0^1\widetilde{b}_{\alpha }(x'-y)\overline{\wal_k(y)}\, dy\right|\\
&\leq \int_0^x \left|\widetilde{b}_{\alpha}(x-y)-\widetilde{b}_{\alpha}(x'-y)\right|\, dy
+ \int_x^{x'} \left|\widetilde{b}_{\alpha}(x-y)-\widetilde{b}_{\alpha}(x'-y)\right|\, dy\\
&\qquad+ \int_{x'}^1 \left|\widetilde{b}_{\alpha}(x-y)-\widetilde{b}_{\alpha}(x'-y)\right|\, dy\\
&\leq x\max_{y\in [0,x]}|b_{\alpha}(x-y)-b_{\alpha}(x'-y)|+
(x'-x)\max_{y\in [x,x']}\left( |b_{\alpha }(y-x)|+|b_{\alpha }(x'-y)|\right) \\
&\qquad +(1-x')\max_{y\in [x',1]}|b_{\alpha}(y-x)-b_{\alpha}(y-x')|\\
&< x\epsilon/2 + 2\delta_2 \max_{z\in [0,1]} |b_{\alpha}(z)| + (1-x')\epsilon/2 \\
&< \epsilon,
\end{align*}
which implies the continuity of $h_1$.
\end{proof}

The following result follows now from
the above lemma, Lemma~\ref{lem:W(b_r)=I}
and \eqref{eq:WB_formula}.
\begin{proposition}\label{prop:exp_a>v}
Assume $\alpha \geq v$. Then for $f \in \Sob$ we have
\begin{align*}
\widehat{f}(k)
&= \sum_{i=v}^{\alpha} (-1)^{i} \Iextra{i-v}{k} \int_0^1 f^{(i)}(x) \, dx \notag \\
&\quad+ (-1)^{\alpha} \int_0^1 f^{(\alpha)}(x) (\Wextra{\alpha-v}{k}{x} - \Iextra{\alpha-v}{k}) \, dx.
\end{align*}
\end{proposition}

Now we treat the case $\alpha <v$.
Note that $\overline{\wal_{\summore{\alpha}}(\cdot)} \Wnorm{\sumless{\alpha}}$
is continuous since
$\Wnorm{\sumless{\alpha}}$ takes zero value on the set where
$\wal_{\summore{\alpha}}(\cdot)$ is not continuous.
In the same way as the case $\alpha \geq v$,
we have the following.

\begin{proposition}\label{prop:exp_a<v}
Assume $\alpha < v$. Then we have 
\begin{align*}
- \int_0 ^1\widetilde{b}_{\alpha}(x-y)\overline{\wal_k(y)} \, dy
= \overline{\wal_{\summore{\alpha}}(x)} \W{\sumless{\alpha}}{x}.
\end{align*}
In particular, for $f \in \Sob$ we have
$$
\widehat{f}(k)=(-1)^{\alpha} \int_0^1 f^{(\alpha)}(x)\overline{\wal_{\summore{\alpha}}(x)} \W{\sumless{\alpha}}{x} \, dx.
$$
\end{proposition}

\subsection{Upper bound on the Walsh coefficients of functions in Sobolev spaces}
In this subsection, we give a bound on the Walsh coefficients
of functions in $\Sob$.

By Propositions~\ref{prop:exp_a>v} and \ref{prop:exp_a<v}, for $f \in \Sob$ we have
\begin{align*}
|\widehat{f}(k)|
&\leq \sum_{i=v}^{\alpha} |\Iextra{i-v}{k}| \left|\int_0^1 f^{(i)}(x) \, dx \right|
+ N_\alpha \int_0^1 |f^{(\alpha)}(x)| \, dx,
\end{align*}
where $N_\alpha = \|\Wextranorm{\alpha-v}{k} - \Iextra{\alpha-v}{k}\|_{L^\infty}$
if $\alpha \geq v$ and 
$N_\alpha = \Wnormof{\sumless{\alpha}}{\infty}$ otherwise.
Thus, by Propositions \ref{prop:Wnorm-infty} and \ref{prop:Wnorm-dyadic}
and Lemmas \ref{lem:Wextra_sup_av}, \ref{lem:Iextra_sup} and \ref{lem:extra_sup-dyadic},
we have the following.
\begin{theorem}\label{thm:Sob-Walsh-bound}
Let $\alpha$ and $k$ be positive integers.
Assume $f \in \Sob$.
If $b>2$, we have
\begin{align*}
|\widehat{f}(k)|
&\leq \sum_{i=v}^{\alpha} \left|\int_0^1 f^{(i)}(x) \, dx\right| \frac{b^{-\muper{i}{k}}}{\mb^{i}} \left(1 + \Cv\right) \notag \\
&\qquad+ \int_0^1 |f^{(\alpha)}(x)| \, dx \frac{b^{-\muper{\alpha}{k}}}{\mb^{\alpha}} \left(\Mb + \Cv\right),
\end{align*}
and if $b=2$, we have
\begin{align*}
|\widehat{f}(k)|
\leq \sum_{\substack{v \leq i \leq \alpha \\ i = v \text{ mod $2$}}}
\left|\int_0^1 f^{(i)}(x) \, dx\right| \frac{2^{-\muper{i}{k}}}{2^{i}} + \int_0^1 |f^{(\alpha)}(x)| \, dx \frac{2^{-\muper{\alpha}{k}}}{2^{\alpha-1}},
\end{align*}
where for $v>\alpha$ the empty sum $\sum_{i=v}^\alpha$
is defined to be $0$.
\end{theorem}

For an integer $i$ with $v \leq i \leq \alpha$,
$\muper{i}{k} \geq \mu_\alpha(k)$ holds for all $k \in \bN$ 
by the definitions of $\muper{i}{k}$ and $\mu_\alpha(k)$.
Thus, applying H\"{o}lder's inequality to Theorem~\ref{thm:Sob-Walsh-bound}, we obtain the following corollary.
\begin{corollary}\label{cor:bound-Sob-with-norm}
Let $\alpha$ and $k$ be positive integers.
Then, for all $f \in \Sob$, we have
\begin{align*}
|\widehat{f}(k)| \leq b^{-\mu_{\alpha}(k)} C_{b, \alpha, q} \|f\|_{p, \alpha},
\end{align*}
where
$\|f\|_{p,\alpha}
:= \left(\sum_{i=0}^\alpha \left|\int_0^1 f^{(i)}(x) \, dx \right|^p + \int_0^1 |f^{(\alpha)}(x)|^p \, dx \right)^{1/p}$,
where $1 \leq p, q \leq \infty$ are real numbers with $1/p + 1/q =1$, and
$$
C_{b, \alpha, q}
:= \left(\sum_{i=1}^{\alpha} \frac{1}{\mb^{iq}}\left(1+\frac{b\mb}{b-\Mb}\right)^q + \frac{1}{\mb^{\alpha q}}\left(\Mb +\frac{b\mb}{b-\Mb}\right)^q \right)^{1/q}
$$
for $b > 2$ and
$C_{2, \alpha, q}
:= (\sum_{i=1}^{\alpha} 2^{-iq} + 2^{-(\alpha-1) q} )^{1/q}$
for $b=2$.
\end{corollary}

\begin{remark}
This corollary can be generalized to tensor product spaces,
for which the reproducing kernel is just the product of the one-dimensional kernel,
as \cite[Section~14.6]{Dick2010dna}.
\end{remark}

\begin{remark}
As shown in Proposition \ref{prop:exp_a<v},
we can apply Theorem~\ref{thm:Walsh-bound-normal} to functions in $\Sob$
as well as Corollary~\ref{cor:bound-Sob-with-norm}.
We can choose the suitable bound in accordance with the situation, such as what kind of norm we need.
\end{remark}

\section{The Walsh coefficients of smooth periodic functions}\label{sec:periodic}
As in \cite{Dick2009dwc},
we consider a subset of the previous reproducing kernel Hilbert space,
namely, let $\Sobper$ be the space of all functions $f \in \Sob$
which satisfy the condition
$\int_0^1 f^{(i)}(x) \, dx = 0$ for $0 \leq i < \alpha$.
This space also has a reproducing kernel, which is given by
$$
\mathcal{K}_{\alpha, \mathrm{per}}(x,y)
= b_\alpha(x) b_\alpha(y) + (-1)^{\alpha+1} \widetilde{b}_{2\alpha}(x-y),
$$
and the inner product is given by
$$
\innerprodper{f}{g} = \int_0^1 f^{(\alpha)}(x) g^{(\alpha)}(x) \, dx,
$$
see \cite[(10.2.4)]{Wahba1990smo}.
We also have the representation
\begin{align*}
f(y)
&= \innerprodper{f}{\mathcal{K}_{\alpha, \mathrm{per}}(\cdot,y)}\\
&= \int_0^1 f^{(\alpha)}(x)\, dx \, b_\alpha(y) + (-1)^{\alpha+1}\int_0^1 f^{(\alpha)}(x) \widetilde{b}_{\alpha}(x-y)\, dx
\end{align*}
and
\begin{align*}
\widehat{f}(k)
=\int_0^1 f^{(\alpha)}(x)\, dx \, \widehat{b_\alpha}(k) + (-1)^{\alpha+1}\int_0^1 f^{(\alpha)}(x) \int_0^1\widetilde{b}_{\alpha}(x-y)\overline{\wal_k(y)}\, dx\, dy.
\end{align*}

By the condition $\int_0^1 f^{(i)}(x) \, dx = 0$ for $0 \leq i < \alpha$ and Propositions \ref{prop:exp_a>v} and \ref{prop:exp_a<v},
we have the following.
\begin{lemma}
Let $\alpha$ and $k$ be positive integers.
Assume $f \in \Sobper$.
If $\alpha \geq v$, then we have
\begin{align*}
\widehat{f}(k)
= (-1)^{\alpha} \int_0^1 f^{(\alpha)}(x) \Wextra{\alpha-v}{k}{x} \, dx.
\end{align*}
If $\alpha < v$, then we have
$$
\widehat{f}(k)
=(-1)^{\alpha} \int_0^1 f^{(\alpha)}(x)\overline{\wal_{\summore{\alpha}}(x)} \W{\sumless{\alpha}}{x} \, dx.
$$
\end{lemma}

This lemma, Propositions \ref{prop:Wnorm-infty} and \ref{prop:Wnorm-dyadic} and
Lemmas \ref{lem:Iextra_sup} and \ref{lem:extra_sup-dyadic} imply the following bound.
\begin{theorem}\label{thm:bound-Sobper}
Let $\alpha$ and $k$ be positive integers.
Assume $f \in \Sobper$.
If $b>2$, then we have
\begin{align*}
|\widehat{f}(k)|
&\leq \int_0^1 |f^{(\alpha)}(x)| \, dx \frac{b^{-\muper{\alpha}{k}}}{\mb^{\alpha}} \Mb\left(1 + \Cv\right).
\end{align*}
If $b=2$, then we have 
\begin{align*}
|\widehat{f}(k)|
&\leq \int_0^1 |f^{(\alpha)}(x)| \, dx \frac{b^{-\muper{\alpha}{k}}}{2^{\alpha-1}}.
\end{align*}
\end{theorem}

\section*{Acknowledgment}
The work of the authors was supported by the Program for Leading Graduate Schools, MEXT, Japan.
The first author was partially supported by
Grant-in-Aid for JSPS Fellows Grant number 15J05380.

\bibliography{SuzukiYoshiki}

\end{document}